\documentclass[11pt,a4paper,leqno]{article}
\usepackage[margin=1.1in]{geometry}
\usepackage{amsmath,amsthm,amssymb,enumerate, bbm ,graphicx,color,caption,upgreek, float, wasysym, subcaption,booktabs,longtable, appendix,subcaption,graphics, pdfpages,rotating,mathtools,subcaption,array,supertabular}
\usepackage[justification=justified,singlelinecheck=false]{caption}

\definecolor{donkergroen}{RGB}{46,148,0}
\usepackage[hidelinks]{hyperref} 
\definecolor{donkerrood}{RGB}{204,0,0}

\definecolor{blauw}{RGB}{61,158,255}
\definecolor{donkerblauw}{RGB}{0,0,255}
\definecolor{donkergroen}{RGB}{46,148,0}
\definecolor{donkerrood}{RGB}{204,0,0}

\usepackage{abstract}

\makeatletter
\newif\if@borderstar
\def\bordermatrix{\@ifnextchar*{%
\@borderstartrue\@bordermatrix@i}{\@borderstarfalse\@bordermatrix@i*}%
}
\def\@bordermatrix@i*{\@ifnextchar[{\@bordermatrix@ii}{\@bordermatrix@ii[()]}}
\def\@bordermatrix@ii[#1]#2{%
\begingroup
\m@th\@tempdima8.75\p@\setbox\z@\vbox{%
\def\cr{\crcr\noalign{\kern 2\p@\global\let\cr\endline }}%
\ialign {$##$\hfil\kern 2\p@\kern\@tempdima & \thinspace %
\hfil $##$\hfil && \quad\hfil $##$\hfil\crcr\omit\strut %
\hfil\crcr\noalign{\kern -\baselineskip}#2\crcr\omit %
\strut\cr}}%
\setbox\tw@\vbox{\unvcopy\z@\global\setbox\@ne\lastbox}%
\setbox\tw@\hbox{\unhbox\@ne\unskip\global\setbox\@ne\lastbox}%
\setbox\tw@\hbox{%
$\kern\wd\@ne\kern -\@tempdima\left\@firstoftwo#1%
\if@borderstar\kern2pt\else\kern -\wd\@ne\fi%
\global\setbox\@ne\vbox{\box\@ne\if@borderstar\else\kern 2\p@\fi}%
\vcenter{\if@borderstar\else\kern -\ht\@ne\fi%
\unvbox\z@\kern-\if@borderstar2\fi\baselineskip}%
\if@borderstar\kern-2\@tempdima\kern2\p@\else\,\fi\right\@secondoftwo#1 $%
}\null \;\vbox{\kern\ht\@ne\box\tw@}%
\endgroup
}
\makeatother

\makeatletter 
\newcommand\mynobreakpar{\par\nobreak\@afterheading} 
\makeatother

\newcommand{\C}{\mathbb{C}}
\newcommand{\R}{\mathbb{R}}

\usepackage[english]{babel}

\newtheorem{theorem}{Theorem}[section]
\newtheorem{lemma}[theorem]{Lemma}
\newtheorem{claim}[theorem]{Claim}
\newtheorem{proposition}[theorem]{Proposition}
\newtheorem{corollary}[theorem]{Corollary}

\theoremstyle{definition}

\newtheorem*{examp*}{Example}

\theoremstyle{plain}

\newfloat{Algorithm}{!hbt}{alg}

\newcounter{thm}[section]

\title{{\large \textbf{PARTITION FUNCTIONS AND A GENERALIZED COLORING-FLOW DUALITY FOR EMBEDDED GRAPHS}}}
\author{{\normalsize Bart Litjens\thanks{University of Amsterdam. Email: \texttt{b.m.litjens@uva.nl }} \hspace{1mm}and Bart Sevenster \thanks{University of Amsterdam. Email: \texttt{blsevenster@gmail.com}}}\footnote{Korteweg-De Vries Institute for Mathematics, University of Amsterdam, Amsterdam, The Netherlands. The research leading to these results has received funding from the European Research Council under the European Union's Seventh Framework Programme (FP7/2007-2013) / ERC grant agreement n$^{\circ}$ 339109.}}
\date{\vspace{-5ex}}
\begin{document}
\maketitle

\noindent {\small \textbf{Abstract.} Let $G$ be a finite group and $\chi: G \rightarrow \C$ a class function. Let $H = (V,E)$ be a directed graph with for each vertex a cyclic order of the edges incident to it. The cyclic orders give a collection $F$ of faces of $H$. Define the partition function $P_{\chi}(H) \coloneqq \sum_{\kappa: E \rightarrow G}\prod_{v \in V}\chi(\kappa(\delta(v)))$, where $\kappa(\delta(v))$ denotes the product of the $\kappa$-values of the edges incident with $v$ (in cyclic order), where the inverse is taken for edges leaving $v$. Write $\chi = \sum_{\lambda}m_{\lambda}\chi_{\lambda}$, where the sum runs over irreducible representations $\lambda$ of $G$ with character $\chi_{\lambda}$ and with $m_{\lambda} \in \C$ for every $\lambda$. When $H$ is connected, it is proved that $P_{\chi}(H) = |G|^{|E|}\sum_{\lambda}\chi_{\lambda}(1)^{|F|-|E|}m_{\lambda}^{|V|}$, where $1$ is the identity element of $G$. Among the corollaries, a formula for the number of nowhere-identity $G$-flows on $H$ is derived, generalizing a result of Tutte. We show that these flows correspond bijectively to certain proper $G$-colorings of a covering graph of the dual graph of $H$. This correspondence generalizes coloring-flow duality for planar graphs.}


\section{Introduction}\label{section:Introduction}

Let $H = (V,E)$ be a directed graph (not necessarily simple) with for each vertex $v$ a cyclic order $\phi_{v}$ of the edges incident to $v$. A loop appears twice in the cyclic order. For a vertex $v$, let $d(v)$ denote the total degree of $v$. If $v$ is a vertex that is not isolated, we fix an edge $e_v^{1}$ adjacent to $v$ and let $e_v^{i+1}$ denote the edge incident to $v$ that in the cyclic order of $v$ comes after $e_v^{i}$, for $i = 1,...,d(v)-1$. For $v \in V$, let $\sigma_{v}: \{e_v^{1},...,e_v^{d(v)}\} \rightarrow \{\pm 1\}$ be the function that takes the value $-1$ on edges that are directed \textit{outwards} at $v$ and $1$ on the edges that are directed \textit{inwards} at $v$. A \textit{face} of $H$ is given by a closed walk with the following properties: the start vertex is the end vertex, successive edges $uv, vw$ (with $u,v,w \in V$) in the walk are successive in the cyclic order $\phi_v$ at $v$, and each edge $uv$ in the walk is traversed at most once going from $u$ to $v$ and at most once going from $v$ to $u$.\\
\indent Let $G$ be a finite group and $\chi: G \rightarrow \C$ a \textit{class function} of $G$, i.e., a function that is constant on conjugacy classes of $G$. Define the \textit{partition function} $P_{\chi}$ with respect to $\chi$ and evaluated at $H$ to be 
\begin{equation}\label{equation:partfunction}
P_{\chi}(H) \coloneqq \sum_{\kappa: E \rightarrow G}\prod_{v \in V}\chi(\kappa(\phi_{v}^{\sigma_{v}})),
\end{equation}
where the sum runs over all maps $\kappa: E \rightarrow G$ and where $\kappa(\phi_{v}^{\sigma_{v}}) \coloneqq \prod_{i=1}^{d(v)}\kappa(e_v^i)^{\sigma_{v}(e_v^i)}$. For a vertex $v$, the choice for $e_v^1$ is irrelevant. Indeed, if $g_1,...,g_d \in G$ for some $d \geq 1$, then the product $g_2\dots g_d$ is the conjugate of $g_1\dots g_d$ by $g_1$ and hence the class function $\chi$ agrees on both products. This shows that the partition function is well-defined. Let $f$ be an edge of $H$. If $f$ is a loop at the vertex $v$ and appears in positions $i$ and $j$ in $\phi_v$, then $\sigma_v(e_v^i) = -\sigma_v(e_v^j)$. If $f = uv$, with $u$ and $v$ different vertices, then $\sigma_v(f) = -\sigma_u(f)$. Let $H_f$ be the directed graph obtained from $H$ by reversing the direction of $f$. Consider a map $\kappa: E \rightarrow G$ on $H$. The map $\kappa_f: E \rightarrow G$ on $H_f$ defined by $\kappa_f(e) = \kappa(e)$ for all $e \neq f$, and $\kappa_f(f) = -\kappa(f)$, has the same weight as $\kappa$. This shows that the partition function depends only on the underlying undirected graph of $H$.\\
\indent It is a standard fact from representation theory that one can write $\chi = \sum_{\lambda}m_{\lambda}\chi_{\lambda}$, where the sum runs over \textit{irreducible representations} $\lambda$ of $G$ with \textit{character} $\chi_{\lambda}$ and with $m_{\lambda} \in \C$ for every $\lambda$. In this paper we give a closed formula for the partition function (\ref{equation:partfunction}) and derive three corollaries by taking special cases for $H$ and $\chi$. In Section \ref{section:formula} it will be proved that if $H$ is connected, it holds that
\begin{equation}\label{equation:partfunction2}
P_{\chi}(H) = |G|^{|E|}\sum_{\lambda}\chi_{\lambda}(1)^{|F|-|E|}m_{\lambda}^{|V|},
\end{equation}
where $F$ denotes the collection of faces of $H$ and $1$ is the identity element of $G$. For every $\lambda$, the coefficient $\chi_{\lambda}(1)$ equals the dimension of the representation $\lambda$. \\
\indent For nonnegative integers $k$ and $g$, the Generalized Frobenius Formula \cite[Thm. \hspace{-2mm}$3$]{zagier04} counts homomorphisms satisfying a certain constraint from the fundamental group of a $k$-fold punctured closed orientable surface of genus $g$ to $G$. In Section \ref{section:counthom} formula (\ref{equation:partfunction2}) with a special choice of class function and graph is used to give a new proof of the case $k=1$ of the Generalized Frobenius Formula.\\
\indent In this paper we will consider nowhere-identity $G$-flows on $H$. These flows extend the notion of nowhere-zero flows in a finite abelian group. The latter were shown \cite{tutte54} to be counted by a polynomial in the size of the group. In order to define $G$-flows for a nonabelian group $G$, the cyclic orders of edges around vertices as defined at the beginning of the introduction, are used.\\
\indent For the character $\chi_{\text{reg}}$ of the \textit{regular representation} of $G$, the partition function $P_{\chi_{\text{reg}}}(H)$ counts the number of \textit{G-flows} on $H$. In Section \ref{section:countflow} a formula for the number of nowhere-identity $G$-flows of $H$ is deduced from equation (\ref{equation:partfunction2}) by inclusion-exclusion. Such a formula was independently found in \cite{goodall16} by essentially the same method as ours, by deriving formula (\ref{equation:partfunction2}) for the special case $\chi = \chi_{\text{reg}}$. As the authors of this paper also mention, in the combinatorial literature there has been little attention given to $G$-flows when $G$ is nonabelian beyond the work of \cite{devos00}.\\
\indent There is a well-known correspondence between a graph $H$ having a cyclic order of edges around each vertex and a cellular embedding of $H$ in a closed orientable surface $S$, variously attributed to Heffter, Ringel and Edmonds (see Theorem \ref{theorem:ehr} below). In Section \ref{section:duality} it is proved that nowhere-identity $G$-flows of $H$ in $S$ correspond bijectively to certain proper $G$-colorings of a (finite) covering graph of $H^*$, the geometric dual of $H$ in $S$. This correspondence generalizes coloring-flow duality for planar graphs, first proved by Tutte in \cite{tutte54}. Our generalization differs from that of \cite[Corollary $1.5$]{devos05}, which gives an ``approximate coloring-flow duality" for $\R$-valued flows and tensions.

\section{The partition function}\label{section:partitionfunction}

\subsection{Preliminaries on representation theory}\label{subsection:repth}

In this section we review some of the representation theory of finite groups. More details and proofs of the results stated may be found in the first two chapters of \cite{serre77}. Furthermore, some notation is introduced that is used in Sections  \ref{section:formula} and \ref{section:countflow}.\\
\indent Let $G$ be a finite group. A \textit{representation} of $G$ is a homomorphism $\rho: G \rightarrow \text{GL}(V)$, where $V$ is a finite-dimensional $\C$-vector space, called the \textit{representation space}, and $\text{GL}(V)$ is the group of invertible linear maps from $V$ to itself, which can be identified with the invertible square matrices of size $\text{dim}(V)$. In this paper all representations $\rho$ are assumed to be \textit{unitary}, i.e., $\rho(g)$ is a unitary matrix for all $g \in G$.\\
\indent If $\rho$ and $\psi$ are representations of $G$ with representation spaces $V$ and $W$ respectively, then a \textit{G-equivariant map} from $V$ to $W$ is a linear map $\phi: V \rightarrow W$ such that $\phi(\rho(g) \cdot v) = \psi(g) \cdot \phi(v)$ for all $g \in G$ and $v \in V$. The representations $\rho$ and $\psi$ are \textit{isomorphic} if there exists a bijective $G$-equivariant map from $V$ to $W$. \\
\indent A representation $\rho: G \rightarrow \text{GL}(V)$ is called \textit{irreducible} if there is no nontrivial subspace $U$ of $V$ such that $\rho(g) \cdot U = U$ for all $g \in G$. Every representation is isomorphic to a direct sum of irreducible representations. The number of pairwise nonisomorphic irreducible representations of $G$ equals the number of conjugacy classes of $G$.\\
\indent Associated to a representation $\rho$ is the function $\chi: G \rightarrow \C$, called the \textit{character} of $\rho$, defined for $g \in G$ by $\chi(g) \coloneqq \text{tr}(\rho(g))$, where \text{tr} stands for the trace. Two representations are isomorphic if and only if their characters are the same. The character of a representation is a class function and the complex space of class functions is equipped with an inner product $\langle \cdot, \cdot \rangle$ defined by
\[
\langle \chi_1,\chi_2\rangle \coloneqq \frac{1}{|G|}\sum_{g \in G}\chi_1(g)\chi_2(g^{-1}),
\]
for class functions $\chi_1$ and $\chi_2$. The characters of the irreducible representations form an orthogonal basis with respect to this inner product. The orthogonality relation below is key for our proof of formula (\ref{equation:partfunction2}). In the following, let $\delta_{ab}$ denote the \textit{Kronecker delta function} equal to $1$ if $a=b$ (when integers) or $a$ is isomorphic to $b$ (when representations), and zero otherwise.
\begin{theorem}[Schur orthogonality]\label{theorem:schur}
Let $\rho$ and $\psi$ be irreducible (unitary) representations of $G$ with characters $\chi_{\rho}$ and $\chi_{\psi}$, then
\[
\sum_{g \in G}\rho(g^{-1})_{n,m}\cdot\psi(g)_{n',m'} = \delta_{nm'}\delta_{mn'}\delta_{\rho\psi}\frac{|G|}{\chi_{\rho}(1)},
\]
for all indices $n,m,n',m' = 1,...,\chi_{\rho}(1)$.
\end{theorem}
\noindent It can be proved that the dimension $\chi_{\rho}(1)$ divides $|G|$.\\
\indent An important representation is the \textit{regular representation} $\rho_{\text{reg}}$. Let $\C[G]$ be the \textit{group algebra} of $G$, defined as the vector space over $\C$ with basis $G$. The product in $G$ defines by linear extension the structure of an algebra on $\C[G]$. Then the regular representation is defined by 
$\rho_{\text{reg}}: G \rightarrow \text{GL}(\C[G]), g \mapsto L_g$,where $L_g$ is the map that denotes left multiplication by $g$ in $\C[G]$. Its character $\chi_{\text{reg}}$ is $|G|$ on the identity element of $G$ and is zero elsewhere. It decomposes as $\sum_{\lambda}\chi_{\lambda}(1)\chi_{\lambda}$, where the sum runs over all irreducible representations $\lambda$ of $G$.

\subsection{A formula}\label{section:formula}

In this section we derive formula (\ref{equation:partfunction2}) for the partition function $P_{\chi}(H)$ of $H$ defined in equation (\ref{equation:partfunction}). This involves a direct computation that relies mostly on Schur orthogonality and the observation that the faces of $H$ determine which matrix coefficients give non-zero contributions. The proof is essentially the same as in \cite{goodall16}.\\
\indent With the notation and assumptions introduced at the beginning of Section \ref{section:Introduction}, we have the following result:
\begin{theorem}\label{theorem:partfunctionformula}
If $H$ is connected, the partition function $P_{\chi}(H)$ equals
\[
P_{\chi}(H) = |G|^{|E|}\sum_{\lambda}\chi_{\lambda}(1)^{|F|-|E|}m_{\lambda}^{|V|},
\]
where the sum runs over all irreducible representations $\lambda$ of $G$.
\end{theorem}
\begin{proof}
Let $\Lambda \coloneqq \{\lambda \mid m_{\lambda} \neq 0\}$. We calculate that
\begin{align}\label{align:berek}
P_{\chi}(H) &= \sum_{\kappa: E \rightarrow G}\prod_{v \in V}\chi(\kappa(\phi_{v}^{\sigma_{v}}))\nonumber\\
&= \sum_{\kappa: E \rightarrow G}\prod_{v \in V}\sum_{\lambda}m_{\lambda}\chi_{\lambda}(\kappa(\phi_{v}^{\sigma_{v}}))\nonumber\\
&= \sum_{\kappa: E \rightarrow G}\sum_{\mu: V \rightarrow \Lambda}\prod_{v \in V}m_{\mu(v)}\chi_{\mu(v)}(\kappa(\phi_{v}^{\sigma_{v}})).
\end{align}
For a fixed $\mu: V \rightarrow \Lambda$, we get a contribution
\begin{align}\label{align:berek2}
\sum_{\kappa: E \rightarrow G}\prod_{v \in V}m_{\mu(v)}\chi_{\mu(v)}(\kappa(\phi_{v}^{\sigma_{v}})) &= \sum_{\kappa: E \rightarrow G}\prod_{v \in V}m_{\mu(v)}\chi_{\mu(v)}(\prod_{i=1}^{d(v)}\kappa(e_v^i)^{\sigma_{v}(e_v^i)})\nonumber\\
&= \sum_{\kappa: E \rightarrow G}\prod_{v \in V}m_{\mu(v)}\text{tr}(\prod_{i=1}^{d(v)}\mu(v)(\kappa(e_v^i)^{\sigma_{v}(e_v^i)}))\nonumber\\
&= \sum_{\kappa: E \rightarrow G}\prod_{v \in V}m_{\mu(v)}\sum_{j_1,...,j_{d(v)}=1}^{\chi_{\mu(v)}(1)}\prod_{i=1}^{d(v)}(\mu(v)(\kappa(e_v^i)^{\sigma_{v}(e_v^i)}))_{j_{i},j_{i+1}},
\end{align}
where we used that the representation $\mu(v)$ is a homomorphism in the second equality, expanded the trace of a product of matrices in the third equality and where the subindices $i$ of $j_{i}$ are taken modulo $d(v)$ for all $i$.\\
\indent Let $e \in E$ be an edge that is directed from $u$ to $v$. Then in expression (\ref{align:berek2}) there is a term of the form
\begin{equation}\label{equation:uitdrukking}
\sum_{a,b = 1}^{\chi_{\mu(u)}(1)}\sum_{c,d = 1}^{\chi_{\mu(v)}(1)}\sum_{\kappa: \{e\} \rightarrow G}(\mu(u)(\kappa(e)^{-1}))_{a,b}\cdot(\mu(v)(\kappa(e)))_{c,d}.
\end{equation}
Hence, by Theorem \ref{theorem:schur} and by connectivity of $H$, we can assume that $\mu: V \rightarrow \Lambda$ is a constant function, as otherwise expression (\ref{equation:uitdrukking}) is zero.\\
\indent Let $\mathcal{F}$ be a face of $H$ with facial walk $W_{\mathcal{F}}$ consisting of the edges $e_1,...,e_r$ (in this order) for some $r \geq 1$. We may assume that $e_i$ is directed from $v_i$ to $v_{i+1}$ for $i=1,...,r-1$ and that $e_{r}$ is directed from $v_{r}$ to $v_{1}$. Fix $\lambda$ and let $\mu: V \rightarrow \Lambda$ be the constant function defined by $\mu(v) = \lambda$ for all $v \in V$. In expression (\ref{align:berek2}) the following term occurs
\begin{equation}\label{equation:indicess}
\sum_{j_1,...,j_{4r} = 1}^{\chi_{\lambda}(1)}\sum_{\kappa: \{e_1,...,e_r\} \rightarrow G}\prod_{i=1}^{r}(\lambda(\kappa(e_i)^{-1}))_{j_{4i},j_{4i+1}}(\lambda(\kappa(e_i)))_{j_{4i+2},j_{4i+3}},
\end{equation}
where the subindices $i$ of $j_i$ are taken modulo $4r$ for all $i$. Since $\mathcal{F}$ is a facial walk, we have that $j_{4i} = j_{4i-1}$ for all $i$. If the term (\ref{equation:indicess}) is nonzero, by Theorem \ref{theorem:schur} we must have that $j_{4i} = j_{4i+3}$ and $j_{4i+1} = j_{4i+2}$ for all $i$. Hence there is an index that is dominant, i.e., that occurs in every matrix coefficient in the above product.\\
\indent For each $\lambda$, applying the above for every face $\mathcal{F}$ yields $|F|$ indices, each ranging from $1$ to $\chi_{\lambda}(1)$. We thus see that the faces of $H$ determine the matrix coefficients that give nonzero contributions to the partition function. For every edge $e$, we need only consider the matrix coefficient whose indices correspond to the faces on either side of $e$. Using expressions (\ref{align:berek}) and (\ref{align:berek2}) we compute
\begin{align*}
P_{\chi}(H) &=  \sum_{\kappa: E \rightarrow G}\sum_{\mu: V \rightarrow \Lambda}\prod_{v \in V}m_{\mu(v)}\chi_{\mu(v)}(\kappa(\phi_{v}^{\sigma_{v}}))\\
&= \sum_{\lambda}\sum_{\kappa: E \rightarrow G}\prod_{v \in V}m_{\lambda}\chi_{\lambda}(\kappa(\phi_{v}^{\sigma_{v}}))\\
&= \sum_{\lambda}m_{\lambda}^{|V|}\chi_{\lambda}(1)^{|F|}\sum_{\kappa: E \rightarrow G}\prod_{e \in E} (\lambda(\kappa(e)^{-1}))_{1,1}(\lambda(\kappa(e)))_{1,1}\\
&= \sum_{\lambda}m_{\lambda}^{|V|}\chi_{\lambda}(1)^{|F|}\prod_{e \in E}\sum_{\kappa: \{e\} \rightarrow G}(\lambda(\kappa(e)^{-1}))_{1,1}(\lambda(\kappa(e)))_{1,1}\\
&= |G|^{|E|}\sum_{\lambda}\chi_{\lambda}(1)^{|F|-|E|}m_{\lambda}^{|V|},
\end{align*}
where in the third equality we used that every choice of indices between $1$ and $\chi_{\lambda}(1)$ gives the same contribution by Theorem \ref{theorem:schur} and in the last equality Theorem \ref{theorem:schur} is used again.
\end{proof}

\section{Consequences of Theorem \ref{theorem:partfunctionformula}}\label{section:consequences}

\subsection{Preliminaries on surfaces and the fundamental group}\label{subsection:surfaces}

We assume some familiarity with the theory of surfaces and the fundamental group. A brief review of the theory of surfaces and the fundamental group is given here. Further information on these topics can be found in \cite{gross87} and \cite{hatcher02}.\\
\indent A \textit{surface} is a connected topological space that is Hausdorff and such that every point has an open neighborhood that is homeomorphic to an open disc in $\R^2$. A \textit{closed surface} is a surface that is compact. If a surface admits a triangulation by oriented triangles such that for any edge of a triangle on the surface the orientations of the neighboring triangles are opposite, it is called \textit{orientable}. A surface is \textit{oriented} if such a triangulation is given. A closed orientable surface is homeomorphic to the sphere with $g$ handles attached to it, for some $g \geq 0$ which is called the \textit{genus}.\\
\indent Let $X$ be a topological space. A \textit{loop} is a continuous function $f$ from the unit interval $[0,1]$ in $\R$ to $X$ such that $f(0) = f(1)$. If $x \in X$, then a loop $f$ is \textit{based at} $x$ if $f(0) = f(1) = x$. Two loops $f_0$ and $f_1$ are said to be \textit{homotopic} if there exists a continuous function $F: [0,1] \times [0,1] \rightarrow X$ such that $F(t,0) = f_0$  and $F(t,1) = f_1$ for all $t \in [0,1]$. \\
\indent Fix $x \in X$. Being homotopic defines an equivalence relation on the set of loops based at $x$. The set of equivalence classes forms a group with respect to concatenation of loops, called the \textit{fundamental group of }X (based at $x$) and denoted by $\pi_1(X,x)$. If $X$ is path-connected then $\pi_1(X,x)$ and $\pi_1(X,y)$ are isomorphic for $y \in X$, which justifies the notation $\pi_1(X)$ in this case.

\subsection{Counting homomorphisms}\label{section:counthom}

In this section a first consequence of Theorem \ref{theorem:partfunctionformula} is presented. Fix a $g \geq 0$ and let $S_g$ be the closed orientable surface of genus $g$. Fix a point $x \in S_g$ and remove it from $S_g$. The remaining surface is denoted $S_{g,1}$ and is called the \textit{punctured surface of genus g}. Its fundamental group can be described by a set of generators subject to one relation
\begin{equation}\label{equation:presentation}
\pi_1(S_{g,1}) = \langle a_1,...,a_g,b_1,...,b_g,c \mid \prod_{i=1}^{g}[a_i,b_i] = c^{-1} \rangle,
\end{equation}
where $[x,y] = xyx^{-1}y^{-1}$ denotes the commutator of $x$ and $y$.\\
\indent Let $G$ be a finite group with identity element $1$ and let $C$ be a conjugacy class of $G$. We define 
\[
A_g(G,C) := |\{(x_1,...,x_g,y_1,...,y_g,z) \in G^{2g} \times C \mid \prod_{i=1}^{g}[x_i,y_i] = z^{-1}\}|.
\]
Then we have that
\[
A_g(G,C) = |\{\phi: \pi_1(S_{g,1}) \rightarrow G \mid \phi(c) \in C\}|,
\]
where the $\phi$ that are counted are assumed to be group homomorphisms and $c$ is the generator in the presentation (\ref{equation:presentation}) of $\pi_1(S_{g,1})$. The parameter $A_g(G,C)$ has a geometric interpretation in terms of coverings of $S_{g,1}$, see the appendix by Zagier in \cite{zagier04}.\\
\indent The following theorem, which is a special case of the Generalized Frobenius Formula \cite[Thm. $3$]{zagier04}, is derived from Theorem \ref{theorem:partfunctionformula}. The proof given in \cite{zagier04} relies on Schur orthogonality and computing the trace of the action by left multiplication with a central element on the group algebra in two different ways. Our proof below only implicitly uses Schur orthogonality, via the proof of Theorem \ref{theorem:partfunctionformula}.

\begin{theorem}
With notation as above, it holds that
\[
A_g(G,C) = |G|^{2g-1}|C|\sum_{\lambda}\frac{\chi_{\lambda}(C)}{\chi_{\lambda}(1)^{2g-1}},
\]
where the sum runs over all irreducible $\lambda$ and where $\chi_{\lambda}(C) = \chi_{\lambda}(h)$ for an $h \in C$.
\end{theorem}
\begin{proof}
Consider the following function
\[
f_{C}: G \rightarrow \C, \hspace{2mm} f_{C}(h) = \begin{cases}\frac{|G|}{|C|} \text{ if }h \in C\\ 0 \text{ otherwise.}\end{cases} 
\]
It is clear that $f_{C}$ is a class function. Write $f_{C} = \sum_{\lambda}m_{\lambda}\chi_{\lambda}$. Then
\[
m_{\lambda} = \frac{1}{|G|}\sum_{g \in G}f_{C}(g)\chi_{\lambda}(g^{-1}) = \chi_{\lambda}(C^{-1}),
\]
where $C^{-1}$ is the conjugacy class consisting of the inverse elements of $C$. Let $H$ be a graph with one vertex and $2g$ loops attached to it in such a way that for odd $i$, the $i$-th loop only intersects the $(i+1)$-th loop (once). This graph can be embedded on $S_g$. Theorem \ref{theorem:partfunctionformula} gives 
\begin{equation}\label{equation:eersteged}
P_{f_{C}}(H) = |G|^{2g}\sum_{\lambda}\chi_{\lambda}(1)^{1-2g}\chi_{\lambda}(C^{-1}).
\end{equation}
On the other hand, by definition of the partition function we have
\begin{align}\label{align:tweedeged}
P_{f_{C}}(H) = &\frac{|G|}{|C|}\cdot |\{(x_1,...,x_g,y_1,...,y_g) \in G^{2g} \mid \prod_{i=1}^{g}[x_i,y_i] \in C\}| \nonumber\\
=& \frac{|G|}{|C|}\cdot |\{(x_1,...,x_g,y_1,...,y_g,z) \in G^{2g} \times C^{-1} \mid \prod_{i=1}^{g}[x_i,y_i] = z^{-1}\}|\nonumber \\
=& \frac{|G|}{|C|}A_{g}(G,C^{-1}).
\end{align}
Setting $C \coloneqq C^{-1}$ and comparing equations (\ref{equation:eersteged}) and (\ref{align:tweedeged}) yields the theorem.
\end{proof}

\subsection{Preliminaries on graph embeddings}\label{subsection:prelimge}

In this section we give some preliminaries on graph embeddings. These preliminaries and the notation are used in Sections \ref{section:countflow} and \ref{subsection:colflow}. More information concerning graph embeddings may be found in \cite{gross87}.\\
\indent Let $H$ be a (undirected) graph and $S$ a closed orientable surface. An \textit{embedding} of $H$ into $S$, which is denoted by a map $i: H \rightarrow S$, is an assignment of points and simple arcs in $S$ to vertices and edges in $H$ respectively, satisfying the following conditions: distinct vertices are associated with distinct points; a simple arc representing an edge has as its endpoints the points assigned to the vertices of the corresponding edge, and no other points on the arc; and, finally, the interiors of distinct simple are disjoint.\\
\indent An embedding $i: H \rightarrow S$ is called \textit{cellular} (or $2$-cell) if all connected components of $S\setminus i(H)$ are homeomorphic to an open disc. The following theorem is often attributed to Edmonds \cite{edmonds60}, Heffter \cite{heffter91} and Ringel \cite{ringel65}.

\begin{theorem}\label{theorem:ehr}
Let $H$ be a connected graph and $S$ a closed orientable surface. Then a cellular embedding $i: H \rightarrow S$ induces for every vertex $v$ a cyclic order of the edges incident to $v$. Conversely, giving for each vertex $v$ of $H$ a cyclic order of edges incident to $v$ defines a cellular embedding of $H$ in a closed orientable surface.
\end{theorem}

\indent Let $H$ be a connected graph with a cyclic order of edges incident with a common vertex. We define the number $g(H)$ to be the genus of the closed orientable surface in which $H$ cellularly embeds by Theorem \ref{theorem:ehr}. This embedding gives a collection $F$ of faces of $H$. It holds that  
\begin{equation}\label{equation:euler}
|V| - |E| + |F| = 2-2g(H).
\end{equation}
The number $|V|-|E|+|F|$ is called the \textit{Euler characteristic} of $H$.\\
\indent Given a cellular embedding $i$ of $H$ into a closed orientable surface $S$, the \textit{dual} (\textit{graph}) $H^*$ of $H$ in $S$ has vertices equal to the faces of $H$ and an edge joins two vertices in $H^*$ if the corresponding faces in $H$ share an edge. The embedding $i$ gives rise to a map $i^*$ that embeds $H^*$ in $S$, called the \textit{dual embedding}.

\subsection{Counting nowhere-identity flows}\label{section:countflow}

Let $G$ be a finite group with identity element $1$. Let $H = (V,E)$ be a connected directed graph with for each vertex $v$ a cyclic order $\phi_v$ of the edges incident to $v$. For each $v \in V$, let $e_v^1$ be an edge incident to $v$ and let $e_v^{i+1}$ be the edge that comes after $e_v^i$ in the cyclic order, for $i = 1,...,d(v)-1$, where $d(v)$ is the total degree of $v$. A \textit{G-flow} on $H$ is a function $\psi: E \rightarrow G$ such that $\prod_{i=1}^{d(v)}\psi(e_v^i)^{\sigma_{v}(e_v^i)} = 1$ (in $G$) for all $v \in V$. A $G$-flow that never takes the identity as value is called a \textit{nowhere-identity G-flow}.\\
\indent When $G$ is abelian the cyclic ordering of edges around vertices of $H$ is not needed to define a $G$-flow of $H$. Tutte proved in \cite{tutte54} that the number of nowhere-zero $G$-flows depends only on the size of $G$. In this section we generalize Tutte's result by deriving a formula that counts the number of nowhere-identity $G$-flows of $H$, where $G$ is any finite group, a result that recently was  found independently by the authors of \cite{goodall16}. This number depends only on the multiset consisting of the dimensions of the irreducible representations of $G$. We first derive a formula for the number of $G$-flows of $H$, denoted by $N_G(H)$.\\
\indent Let $\chi_{\text{reg}}$ be the character of the regular representation of $G$. With notation as in the beginning of the introduction, the partition function is computed to be
\begin{equation}\label{equation:flows1}
P_{\chi_{\text{reg}}}(H) = \sum_{\kappa: E \rightarrow G}\prod_{v \in V}\chi_{\text{reg}}(\kappa(\phi_{v}^{\sigma_{v}})) = |G|^{|V|}\cdot |\{\kappa: E \rightarrow G \mid \kappa(\phi_{v}^{\sigma_{v}}) = 1 \text{ for all }v\}|.
\end{equation}
The condition that $\kappa(\phi_{v}^{\sigma_{v}}) = 1$ for all $v$, is the conditon for $\kappa$ to be a $G$-flow on $H$. Hence $P_{\chi_{\text{reg}}}(H)$ equals $|G|^{|V|}N_{G}(H)$. On the other hand, as $m_{\lambda} = \chi_{\lambda}(1)$ for the character of the regular representation, Theorem \ref{theorem:partfunctionformula} reads
\begin{equation}\label{equation:flows2}
P_{\chi_{\text{reg}}}(H) = |G|^{|E|}\sum_{\lambda}\chi_{\lambda}(1)^{|V|-|E|+|F|} = |G|^{|E|}\sum_{\lambda}\chi_{\lambda}(1)^{2-2g(H)},
\end{equation}
where equation (\ref{equation:euler}) is used in the second equality. Putting together equations (\ref{equation:flows1}) and (\ref{equation:flows2}) yields
\[
N_G(H) = |G|^{|E|-|V|}\sum_{\lambda}\chi_{\lambda}(1)^{2-2g(H)},
\]
for the number of $G$-flows on $H$.\\
\indent Next we derive a formula for the number of nowhere-identity $G$-flows of $H$, denoted by $\tilde{N}_G(H)$. If $A \subset E$ is a subset of the edge set of $H$, let $(V, A)$ denote the subgraph of $H$ where the edges in $E\setminus A$ are removed. Let $\chi$ be a class function of $G$. Write
\[
P_{\chi}(H) = \sum_{B \subset E}\widetilde{P}_{\chi}((V, E\setminus B)), 
\]
where
\[
\widetilde{P}_{\chi}(H) = \sum_{\kappa: E \rightarrow G\setminus\{1\}}\prod_{v \in V}\chi(\kappa(\phi_{v}^{\sigma_{v}})).
\]
Inclusion-exclusion gives
\begin{equation}\label{equation:inclexcl}
\widetilde{P}_{\chi}(H) = \sum_{B \subset E}(-1)^{|E\setminus B|}P_{\chi}((V,B)) = \sum_{B \subset E}(-1)^{|E\setminus B|}\prod_{C}P_{\chi}(C),
\end{equation}
where the product in the last expression runs over components $C$ of the subgraph $(V,B)$ of $H$. In the second equality we used that the partition function is multiplicative with respect to the disjoint union of graphs. Let $V_C, E_C$ and $F_C$ denote the vertex set, edge set and collection of faces of $C$ respectively. Then equation (\ref{equation:inclexcl}) together with Theorem \ref{theorem:partfunctionformula} gives
\begin{equation}\label{equation:nowhere1part}
\widetilde{P}_{\chi}(H) = \sum_{B \subset E}(-1)^{|E\setminus B|}|G|^{|B|}\prod_{C}\sum_{\lambda}\chi_{\lambda}(1)^{|F_C|-|E_C|}m_{\lambda}^{|V_C|},
\end{equation}
where we used that $\prod_{C}|G|^{|E_C|} = |G|^{\sum_{C}|E_C|} = |G|^{|B|}$.

\begin{theorem}\label{theorem:nowhere1flows}
The number $\widetilde{N}_{G}(H)$ of nowhere-identity $G$-flows on $H$ is given by
\[
\widetilde{N}_{G}(H) = \sum_{B\subset E}(-1)^{|E\setminus B|}|G|^{|B|-|V|}\prod_{C}\sum_{\lambda}\chi_{\lambda}(1)^{2-2g(C)},
\]
where the product runs over all components $C$ of the subgraph $(V,B)$ of $H$ and where $2-2g(C)$ denotes the Euler characteristic of $C$.
\end{theorem}
\begin{proof}
Let $\chi_{\text{reg}}$ be the character of the regular representation of $G$. The observation that
\[
\widetilde{P}_{\chi_{\text{reg}}}(H) = |G|^{|V|}\cdot \widetilde{N}_{G}(H)
\]
combined with equation (\ref{equation:nowhere1part}) for $\chi_{\text{reg}}$ yields the theorem.
\end{proof}

\begin{corollary}\label{corollary:karakterflow}
The number $\widetilde{N}_{G}(H)$ of nowhere-identity $G$-flows on $H$ depends only on $H$ and the multiset consisting of dimensions of the irreducible representations of $G$.
\end{corollary}
\begin{proof}
This follows directly from Theorem \ref{theorem:nowhere1flows}.
\end{proof}

The dihedral group of order $8$ and the quaternion group are the smallest example of two nonisomorphic groups whose multisets consisting of dimensions of irreducible representations coincide. DeVos \cite[Lemma $6.1.6$]{devos00} argues directly that the number of nowhere-identity $G$-flows is the same when $G$ is the quaternion group and when $G$ is the dihedral group of order $8$. Note that Corollary \ref{corollary:karakterflow} incorporates the dependence on the size of the group, as $|G| = \sum_{\lambda}\chi_{\lambda}(1)^{2}$, where the sum runs over all irreducible representations $\lambda$ of $G$. If $G$ is abelian then $\chi_{\lambda}(1) = 1$ for every $\lambda$, of which there are $|G|$ many. Hence in this case Theorem \ref{theorem:nowhere1flows} reduces to
\[
\widetilde{N}_{G}(H) = \sum_{B \subset E}(-1)^{|E\setminus B|}|G|^{|B|-|V|+c(B)},
\]
where $c(B)$ denotes the number of connected components of the subgraph $(V,B)$ of $H$.

\section{Coloring-flow duality for embedded graphs}\label{section:duality}

\subsection{Preliminaries on coverings}\label{subsection:coverings}

In this section some preliminaries on coverings are given that, together with the preliminaries on graph embeddings in Section \ref{subsection:prelimge}, are used to prove the duality between flows and colorings for embedded graphs in Section \ref{subsection:colflow}. For more background information the reader is referred to \cite{gross87}. Details concerning the universal covering graph may be found in \cite{serre80}. Moreover, this section serves as a setup of notation that is used in Section \ref{subsection:colflow}.\\
\indent The \textit{neighborhood} of a vertex of a graph is the set of vertices adjacent to that vertex. Let $K$ and $H$ be graphs. A \textit{homomorphism} between $K$ and $H$ is a mapping from the vertex set of $K$ to the vertex set of $H$ such that if two vertices are adjacent in $K$, their images are adjacent in $H$. Let $s: K \rightarrow H$ be a homomorphism. Then $s$ is said to be a \textit{covering} of $H$ if it is surjective, and the neighborhood of each vertex of $K$ is in bijection with the neighborhood of its image in $H$ under $s$. If there is a loop at a vertex, then the vertex itself is in its own neighborhood.\\ 
\indent For every connected graph $H$ that is not a tree there exists a covering $t: T \rightarrow H$ with the following properties: $T$ is an infinite tree and if $r: K \rightarrow H$ is another covering of $H$ with $K$ connected, then there exists a covering $s: T \rightarrow K$ such that $r \circ s = t$. The graph $T$ is called the \textit{universal covering graph} of $H$ and is unique up to isomorphism. If $H$ is a tree, its universal covering graph is defined to be $H$ itself.\\
\indent After choosing a vertex $v$ of $H$, the vertices of $T$ can be described as non-backtracking walks in $H$ starting at $v$. A vertex $w$ in $T$ is adjacent to the vertices that correspond to simple extensions of the walk in $H$ that is associated with $w$. The covering $t: T \rightarrow H$ is then given by sending a vertex in $T$ to the endpoint of the walk in $H$ it corresponds to. 

\subsection{Coloring-flow duality}\label{subsection:colflow}

In this section we prove that flows on graphs that are embedded in a closed orientable surface correspond bijectively to specific tensions of a covering of the dual of that graph, where the particular covering is dependent on the flow. By doing so, we generalize Tutte's coloring-flow duality for planar graphs. We use the notation and terminology as introduced in Sections \ref{subsection:prelimge} and \ref{subsection:coverings}. First we need to define the concept of a local and global tension.\\
\indent As in the previous sections, let $G$ be a finite group with identity element $1$ and let $H = (V, E)$ be a connected directed graph with a cyclic order at each vertex. If $W = (e_1,...,e_n)$ is a \textit{walk} in $H$ for some $n \geq 1$, then for a map $\psi : E \rightarrow G$ we define the \emph{height} of $W$ as $h_{\psi}(W) = \prod_{i = 1}^n \psi(e_i)^{\sigma_{W}(e_i)}$, where $\sigma_{W}(e_i) = +1$ if $e_i$ is traversed in the direction of $e_i$ and $-1$ else. A \textit{local} $G$-\textit{tension} on $H$ is a function $\psi : E \rightarrow G$, such that $h_{\psi}(W) = 1$ for every facial walk $W$. A \textit{global} $G$-\textit{tension} is a function $\psi: E \rightarrow G$ such that $h_{\psi}(W) = 1$ for every closed walk $W$.

\begin{lemma}\label{lemma:lemmer}
There is a $|G|$-to-$1$ correspondence between proper $G$-colorings of $H$ and nowhere-identity global $G$-tensions on $H$.
\end{lemma}
\begin{proof}
Let $\psi$ be a global $G$-tension on $H$, and let $v$ be a vertex of $H$. Then for a vertex $u$ of $H$, there is a path $W_u$ from $v$ to $u$. Let $\kappa : V(H) \rightarrow G$ be given by $\kappa(u) = h_{\psi}(W_u)$. Because $\psi$ is a global $G$-tension, this is independent of the choice of $W_u$ and hence gives a well-defined function. Furthermore, because $\psi$ is nowhere-identity, we have that $\kappa$ is a proper $G$-coloring. For $g \in G$, the coloring $\kappa_g: V(H) \rightarrow G$ given by $\kappa_g(u) = \kappa(u)\cdot g$ also is a proper $G$-coloring and for $h \in G$ we have $\kappa_g = \kappa_h$ if and only if $g = h$.\\
\indent For the other direction, let $\kappa$ be a proper $G$-coloring of $H$. Then the map $\psi: E(H) \rightarrow G$ defined by $\psi(e) = \kappa(u)\kappa(v)^{-1}$, with $e = uv$ and $u,v \in V(H)$, is nowhere-identity as $\kappa$ is a proper $G$-coloring. It is immediately seen to be a global $G$-tension and the colorings $\kappa_g$ defined above all yield the same $\psi$.
\end{proof}

\indent We use Theorem \ref{theorem:ehr} to cellularly embed $H$ in the closed orientable surface $S$ of genus $g(H)$. Let $H^*$ denote the dual of $H$ in $S$. Any map defined on the edges of $H$ naturally induces a map on the edges of $H^*$ and vice versa. The following lemma is well-known, but we include its short proof.

\begin{lemma}\label{lem:loctens}
There is a $1$-to-$1$ correspondence between $G$-flows on $H$ and local $G$-tensions on $H^*$.
\end{lemma}
\begin{proof}
The flow condition at a vertex of $H$ corresponds precisely to the condition that the height of the facial walk in $H^*$ corresponding to that vertex, is equal to $1$.
\end{proof}

\begin{theorem}[Coloring-flow duality for planar graphs]\label{theorem:flow-colplanar}
When $H$ is planar, there is a $|G|$-to-$1$ correspondence between proper $G$-colorings of $H^*$ and nowhere-identity $G$-flows on $H$.
\end{theorem}
\begin{proof}
For planar graphs every closed walk is homotopic to a facial walk. Hence, local $G$-tensions are actually global $G$-tensions, and Lemmas \ref{lemma:lemmer} and \ref{lem:loctens} then give the desired correspondence.
\end{proof}

\indent For non-planar graphs, a local $G$-tension is not necessarily a global $G$-tension. Hence in general there is no coloring-flow duality for embedded graphs. In Theorem \ref{theorem:flowkleur} we will show that a $G$-flow corresponds to a $G$-coloring of a covering of the dual.\\
\indent From now on, a locally bijective homomorphism between connected directed graphs with a cyclic order at each of their vertices, is assumed to respect the directions of the edges and the cyclic orders. Given such a homomorphism $s: K \rightarrow H$ and any $\psi: E(H) \rightarrow G$, the map $\psi$ lifts to a map $\psi_s : E(K) \rightarrow G$, given by $\psi_s(e) \coloneqq \psi(s(e))$ for every $e \in E(K)$. 

\begin{lemma}
Let $s : K \rightarrow H$ be a covering. If $\psi$ is a local $G$-tension on $H$, then $\psi_s$ is a local $G$-tension on $K$. 
\end{lemma}
\begin{proof}
Let $W$ be a walk around a face $\mathcal{F}$ of $K$. Because $s$ respects the local ordering, we know that $W$ maps to a face $s(\mathcal{F})$ of $H$, where $s(W)$ might go around $s(\mathcal{F})$ several times, say $f$ times. Then we have that $h_{\psi_s}(W) = h_{\psi}(s(W)) = h_{\psi}(s(\mathcal{F}))^f = 1$, where the last equality follows from the fact that $\psi$ is a local $G$-tension. 
\end{proof}

Let $\psi$ be a local $G$-tension on $H$ and let $T$ be the universal covering graph of $H$. Fixing a vertex $v \in V(H)$, we can bijectively associate to a vertex of $T$ a non-backtracking walk $W$ in $H$ starting at $v$. Now label this vertex of $T$ with $(e(W), h_\psi(W))$, where $e(W)$ denotes the end vertex of $W$.\\
\indent Let $H_{\psi} \coloneqq T / \mathcal{I}$, where $\mathcal{I}$ is the equivalence relation on the vertices of $T$ in which like-labelled vertices are equivalent and such that the canonical map $T \rightarrow H_{\psi}$ is a locally bijective homomorphism. If a vertex of $T$ has label $(e(W), h_{\psi}(W))$, its image under the map $T \rightarrow H_{\psi}$ is denoted by $[(e(W),h_{\psi}(W))]$.

\begin{lemma}\label{lemma:lemmatje}
With notation as above, let $\ell : H_{\psi} \rightarrow H$ be given by $[(e(W), h_\psi(W))] \rightarrow e(W)$. Then $\ell$ is a covering, and the map $\psi_{\ell}$ is a global $G$-tension on $H_{\psi}$. 
\end{lemma}
\begin{proof}
The map $\ell$ is clearly well-defined and surjective. We show that it is a covering. For a vertex of $H_{\psi}$ labeled by $[(e(W), h_\psi(W))]$, we can extend the walk $W$ in $H$ exactly to the neighbors of $e(W)$. If these neighbors are different from $e(W)$, then they correspond to different vertices in $H_{\psi}$. By our definition of neighborhood (see Section \ref{subsection:coverings}), this also works in case there is a loop at $e(W)$. This shows that the neighborhood of $v$ with label $[(e(W), h_\psi(W))]$ in $H_{\psi}$ is equal to that of $e(W)$ in $H$.\\
\indent To show that $\psi_{\ell}$ is a global $G$-tension, let $W$ be a closed walk in $H_{\psi}$ starting at $[(v,1)]$. If $e(W) =[(e(\ell(W)), h_\psi(\ell(W))]$, then we have that $h_{\psi_{\ell}}(W) = h_\psi(\ell(W))$. Hence for any closed walk starting at $[(v,1)]$ it holds that $h_{\psi_{\ell}}(W) = 1$.\\
\indent If $W$ is a closed walk starting at a vertex $[(u,h)]$, with $u \in V(H)$ and $h \in G$, different from $[(v,1)]$, we can conjugate it with a path $P$ from $[(u,h)]$ to $[(v,1)]$. It is now easy to see that this walk also has $h_{\psi_{\ell}}(W) = 1$, as $1 = h_{\psi_{\ell}}(PWP^{-1}) = h_{\psi_{\ell}}(P) h_{\psi_{\ell}}(W) h_{\psi_{\ell}}(P)^{-1} $. 
\end{proof}

\begin{theorem}\label{theorem:flowkleur}
A nowhere-identity $G$-flow $\psi$ on $H$ gives rise to a proper $G$-coloring of $H^*_{\psi^*}$, where $\psi^*$ is the dual of $\psi$.
\end{theorem}
\begin{proof}
Combine Lemmas \ref{lemma:lemmer}, \ref{lem:loctens} and \ref{lemma:lemmatje}.
\end{proof}

\noindent Theorem \ref{theorem:flowkleur} can be strengthened to a duality type theorem for embedded graphs (Theorem \ref{theorem:duality}) that generalizes coloring-flow duality for planar graphs. In order to do so, we introduce some terminology.\\
\indent Let $\psi$ be a local $G$-tension on $H$ and $s: K \rightarrow H$ a covering, then we say that $s$ is a \emph{global covering with respect to $\psi$} if $\psi_s$ is a global $G$-tension on $K$. A global covering $p: K \rightarrow H$ with respect to $\psi$ is called \textit{minimal} if for every global covering $r: K' \rightarrow H$ with respect to $\psi$, there is a covering $q: K' \rightarrow K$ such that $p \circ q = r$. 

\begin{proposition}\label{proposition:propper}
Let $\psi$ be a local $G$-tension on $H$. The map $\ell: H_{\psi} \rightarrow H$ from Lemma \ref{lemma:lemmatje} is a minimal global covering with respect to $\psi$.
\end{proposition}
\begin{proof}
That $\ell$ is a global covering is a reformulation of Lemma \ref{lemma:lemmatje}. To prove minimality of $\ell$, let $p: K \rightarrow H$ be a global covering and let $W$ be a walk in $K$ starting at a vertex $u$, such that $p(u) = v$. We can label the end vertex of $W$ in $K$ with $(e(p(W)), h_{\psi}(p(W)))$ (this is well-defined because of Lemma \ref{lemma:lemmatje}). Now let $r: K \rightarrow H_{\psi}$ be such that it maps a vertex labeled $(e(p(W)), h_{\psi}(p(W)))$ to $[(e(p(W)), h_{\psi}(p(W)))]$. The map $r$ is surjective, because we can lift any walk in $H$ to a walk in $K$ with the same height, so any labeling giving a vertex $[(e(W), h_{\psi}(W))]$ in $H_{\psi}$, can be obtained as a labeling in $K$. Because the labeling is well-defined, we are done. 
\end{proof}

\indent A \textit{global covering tension} on $H$ consists of a covering $s: K \rightarrow H$ and a global $G$-tension $\psi$ on $K$ with the following properties: for every $e \in E(H)$ it holds that $\psi(e_1) = \psi(e_2)$ for all $e_1,e_2 \in s^{-1}(e)$, and the map $\psi \circ s^{-1}$ on edges of $H$ (which is well-defined by the first property) is a local $G$-tension. A global covering tension with covering $s$ and global $G$-tension $\psi$ is called \textit{minimal} if $s$ is a minimal global covering with respect to $\psi$. 

\begin{theorem}\label{theorem:duality}
A $G$-flow on $H$ gives rise to a minimal global covering tension on $H^*$ and vice versa.
\end{theorem}
\begin{proof}
Let $\psi$ be a $G$-flow on a graph $H$ embedded on a closed orientable surface $S$. By Lemma \ref{lem:loctens} the map $\psi^*$ is a local $G$-tension on $H^*$, the dual of $H$ in $S$. By Lemma \ref{lemma:lemmatje} and Proposition \ref{proposition:propper} we know that the global covering tension consisting of the covering $\ell$ and global $G$-tension $\psi_{\ell}$ is a minimal global covering. Conversely, we clearly have that a minimal global covering gives a local $G$-tension on $H^*$ and hence a $G$-flow on $H$.
\end{proof}

\begin{corollary}
When $H$ is a plane graph, the correspondence in Theorem \ref{theorem:duality} yields coloring-flow duality for planar graphs.
\end{corollary}
\begin{proof}
If $H$ is a plane graph and $\psi$ a $G$-flow of $H$, then $H^*$ is a plane graph and $\psi^*$ is a local $G$-tension on $H^*$. Because all walks are homotopic in a plane graph, we have that $\psi^*$ is a global $G$-tension. Hence, for two walks $W$ and $W'$ in $H^*$, both starting at a vertex $v$ and both ending in a vertex $w$, we have that $h_{\psi^*}(W) = h_{\psi^*}(W')$. In particular, this means that the covering map $H_{\psi^*}^* \rightarrow H^*$ is an isomorphism, showing that a $G$-flow on $H$ gives a global $G$-tension on $H^*$ and vice versa. Applying Lemma \ref{lemma:lemmer} finishes the proof.
\end{proof}

The minimality of the global covering tension found in Theorem \ref{theorem:duality} can also be formulated in terms of a universal property and it seems natural to further investigate the connections with algebraic topology. We also remark that the covering graph $H_{\psi}$ can be interpreted as the derived graph of a voltage graph, see \cite{gross87}. In a future paper we plan to address this interpretation and we want to investigate other properties of $H_{\psi}$. In \cite{NesSam} the authors count tension-continuous mappings, i.e., maps between graphs such that a tension on the codomain lifts to a tension on the domain. It would be interesting to generalize this concept to local-tension-continuous mappings for embedded graphs and to consider the corresponding counting problem.\\

\noindent \textit{Acknowledgements.} The authors would like to thank Sven Polak, Guus Regts and Lex Schrijver for useful discussions and we thank the referees for many helpful suggestions improving the presentation of the paper.


\begin{thebibliography}{10}
\small
 \addcontentsline{toc}{chapter}{References}
\bibliographystyle{alpha}

\bibitem{devos00}M.J.\ DeVos, \textsl{Flows on Graphs}, PhD thesis, Princeton Univ., 2000.

\bibitem{devos05}M.J.\ DeVos, L.\ Goddyn, B.\ Mohar, D.\ Vertigan and X. Zhu, Coloring-flow duality of embedded graphs, \textsl{Transactions of the American Mathematical Society} 357 No. 10 (2005), 3993--4016.

\bibitem{edmonds60}J.\ Edmonds, A combinatorial representation of polyhedral surfaces, \textsl{Notices of the American Mathematical Society} 7 (1960), 646.

\bibitem{goodall16}A.\ Goodall, T.\ Krajewski, G.\ Regts and L.\ Vena, A Tutte polynomial for maps, 2016, ArXiv \url{https://arxiv.org/pdf/1610.04486.pdf}

\bibitem{gross87}J.L.\ Gross and T.W.\ Tucker, \textsl{Topological Graph Theory}, Wiley-Interscience, New York, 1987.

\bibitem{hatcher02}A.\ Hatcher, \textsl{Algebraic Topology}, Cambridge University Press, 2002.

\bibitem{heffter91}L.\ Heffter, \"Uber das Problem der Nachbargebiete, \textsl{Mathematische Annalen} 38(4) (1891), 477--508.

\bibitem{NesSam} J.\ Ne\v{s}et\v{r}il and R.\ \v{S}\'amal,  On tension-continuous mappings, \textsl{European Journal of Combinatorics} 29(4) (2008), 1025--1054.

\bibitem{ringel65}G.\ Ringel, Das Geschlecht des vollst\"andigen paaren Graphen, \textsl{Abhandlungen aus dem Mathematische Seminar der Universit\"at Hamburg} 28 No. 3 (1965), 139--150.

\bibitem{serre77}J.-P.\ Serre, \textsl{Linear Representations of Finite Groups}, Graduate Texts in Mathematics, Vol. 42, Springer, New York, 1977.

\bibitem{serre80}J.-P.\ Serre, \textsl{Trees}, Springer, Berlin, 1980.

\bibitem{tutte54}W.T.\ Tutte, A contribution to the theory of chromatic polynomials, \textsl{Canad. J. Math.} 6 (1954), 80--91.

\bibitem{zagier04}S.K.\ Lando and A.K. Zvonkin, \textsl{Graphs and Surfaces and Their Applications}, Encyclopedia Math. Sci., Vol. 141, Springer, New York, 2013.


\end{thebibliography}
\end{document}